\documentclass[a4paper,12pt]{article}
\usepackage{amsmath, amsfonts, amssymb, amsthm}
\usepackage[english]{babel}

\newcounter{num}[section]

\newenvironment{corollary}
{\refstepcounter{num}%
\bigskip\noindent\nopagebreak[4]{\bf Corollary~\arabic{section}.\arabic{num}. }\it}

\newenvironment{lemma}
{\refstepcounter{num}%
\bigskip\noindent\nopagebreak[4]{\bf Lemma~\arabic{section}.\arabic{num}. }\it}

\newenvironment{remark}
{\refstepcounter{num}%
\bigskip\noindent\nopagebreak[4]{\bf Remark~\arabic{section}.\arabic{num}. }}

\newenvironment{example}
{\refstepcounter{num}%
\bigskip\noindent\nopagebreak[4]{\bf Example~\arabic{section}.\arabic{num}. }}

\newenvironment{theorem}
{\refstepcounter{num}%
\bigskip\noindent\nopagebreak[4]{\bf Theorem~\arabic{section}.\arabic{num}. }\it}

\newcommand{\LL}{{\mathcal{L}}}

\newcommand{\Ss}{{\mathbf{S}}}
\newcommand{\V}{{\mathrm{V}}}

\newcommand{\pr}{{\prime}}

\newcommand{\al}{{\alpha}}

\newcommand{\A}{{\mathcal{A}}}
\newcommand{\M}{{\mathcal{M}}}
\newcommand{\T}{{\mathcal{T}}}
\renewcommand{\P}{{\mathbf{P}}}
\renewcommand{\c}{{\mathbf{c}}}

\newcommand{\Var}{{\mathrm{Var}}}

\begin{document}

\author{Shevlyakov Artem}
\title{On disjunctions of equations over semigroups}

\maketitle

\abstract{A semigroup $S$ is called an equational domain (e.d.) if any finite union of algebraic sets over $S$ is algebraic. For a semigroup $S$ with a finite two-sided ideal we find necessary and sufficient conditions to be an e.d.}

\section{Introduction}

Following~\cite{uniTh_I,uniTh_II}, one can define the notions of an equation and algebraic set for any algebraic structure $\A$ (group, Lie algebra, semigroup, etc.). It allows us to develop algebraic geometry over every algebraic structure. 

Algebraic sets have the common properties which hold in every algebraic structure $\A$. For example, the intersection of an arbitrary number of algebraic sets is always algebraic in any algebraic structure $\A$.

However the union of algebraic sets is not algebraic in general. In~\cite{uniTh_IV} it was defined the notion of an equational domain (e.d.). An algebraic structure $\A$ is an e.d. if any finite union of algebraic sets over $\A$ is always algebraic. Moreover, in~\cite{uniTh_IV} it was proved necessary and sufficient conditions for any group (Lie algebra, associative ring) to be an e.d. For instance, equational domains in the class of commutative associative rings are exactly the rings with no zero-divisors. 

In~\cite{uniTh_IV} the equational domains in the class of groups were described. By this result, the groups of the next classes are e.d.:
\begin{enumerate}
\item free non-abelian groups (proved by G.~Gurevich, one can see the proof in~\cite{makanin});
\item simple non-abelian groups (it follows from~\cite{rhodes}). 
\end{enumerate}

In~\cite{shevl_ED_I} we found necessary and sufficient conditions for a finite ideal-simple semigroup to be an e.d.

The current paper continues the study of~\cite{shevl_ED_I}, and we investigate the properties of e.d. in the class of semigroups with nontrivial ideals. Sections~\ref{sec:basics},~\ref{sec:al_geom}, contains the definitions of semigroup theory and algebraic geometry. 

In Section~\ref{sec:non_simple_semigroups} we prove the following: if a semigroup $S$ is an e.d. and has a completely simple kernel $K$ then $K$ is also an e.d. (Theorem~\ref{th:main_new}).  

The main result of Section~\ref{sec:action_on_ideal} is Theorem~\ref{th:alpha_sim_beta}. By this theorem, any infinite semigroup with a finite ideal is not an equational domain (Corollary~~\ref{cor:about_infinite_semigroups}).

Finally, in Section~\ref{sec:criterion} we prove the criterion, when a semigroup $S$ with the finite kernel $K$ is an e.d. Precisely, in Theorem~\ref{th:criterion} we prove that necessary conditions of Theorems~\ref{th:main_new},~\ref{th:alpha_sim_beta} are sufficient for such semigroup $S$.

\section{Notions of semigroup theory}
\label{sec:basics}

let us give the classic theorem of semigroup theory

\begin{theorem}
\label{th:sushkevic_rees}
For any completely simple semigroup $S$ there exists a group $G$ and sets $I,\Lambda$ such that  $S$ is isomorphic to the set of triples $(\lambda,g,i)$, $g\in G$, $\lambda\in\Lambda$, $i\in I$ with  multiplication defined by
\[
(\lambda,g,i)(\mu,h,j)=(\lambda,gp_{i\mu}h,j),
\]
where $p_{i\mu}\in G$ is an element of a matrix $\P$ such that
\begin{enumerate}
\item $\P$ consists of  $|I|$ rows and $|\Lambda|$ columns;
\item the elements of the first row and the first column equal  $1\in G$ (i.e. $\P$ is {\it normalized}).
\end{enumerate}
\end{theorem}

Following Theorem~\ref{th:sushkevic_rees}, we denote any completely simple semigroup $S$ by $S=(G,\P,\Lambda,I)$. Notice that the cardinality of the set $\Lambda$, ($I$) is equal to the number of minimal right (respectively, left) ideals of a semigroup $S$.

\begin{corollary}
\label{cor:when_is_group}
A completely simple semigroup $S=(G,\P,\Lambda,I)$ is a group iff $|\Lambda|=|I|=1$.
\end{corollary}

The index $\lambda\in\Lambda$, ($i\in I$) of an element $(\lambda,g,i)\in S$ is called {\it the first} (respectively, {\it the second}) index.

\medskip

The minimal ideal (if it exists) of a semigroup $S$ is called the \textit{kernel} and denoted by $Ker(S)$. Clearly, any finite semigroup has the kernel.  Obviously, if $S=Ker(S)$ the semigroup is simple. If $Ker(S)$ is a group then $S$ is said to be a \textit{homogroup}.

Below we consider only semigroups with completely simple kernel $K=(G,\P,\Lambda,I)$. For example, any finite semigroup has the finite kernel of the form $K=(G,\P,\Lambda,I)$.

\section{Notions of algebraic geometry}
\label{sec:al_geom}

All definitions below are deduced from the general notions of~\cite{uniTh_I,uniTh_II}, where the definitions of algebraic geometry were formulated for an arbitrary algebraic structure in the language with no predicates.

Semigroups as algebraic structures are often considered in the language $\LL_0=\{\cdot\}$. However, for a given semigroup $S$ one can add to the language $\LL_0$ the set of constants $\{s|s\in S\}$. We denote the extended language by $\LL_S$, and below we consider all semigroups in such language.

Let $X$ be a finite set of variables  $x_1,x_2,\ldots,x_n$. \textit{An $\LL_S$-term} in variables  $X$ is a finite product of variables and constants $s\in S$. For example, the following expressions $xsy^2x$, $xs_1ys_2x^2$, $x^2yxz$ ($s,s_1,s_2\in S$) are $\LL_S$-terms.

{\it An equation} over $\LL_S$ is an equality of two $\LL_S$-terms $t(X)=s(X)$. {\it A system of equations} over $\LL_S$ ({\it a system} for shortness) is an arbitrary set of equations over $\LL_S$.
 
A point $P=(p_1,p_2,\ldots,p_n)\in S^n$ is a \textit{solution} of a system $\Ss$ in variables $x_1,x_2,\ldots,x_n$, if the substitution $x_i=p_i$ reduces any equation of $\Ss$ to a true equality in the semigroup $S$. The set of all solutions of a system $\Ss$ in a semigroup $S$ is denoted by $\V_S(\Ss)$. A set $Y\subseteq S^n$ is called  {\it algebraic} over the language $\LL_S$ if there exists a system over $\LL_S$ in variables $x_1,x_2,\ldots,x_n$ with the solution set $Y$. 

Following~\cite{uniTh_IV}, let us give the main definition of our paper.

A semigroup $S$ is an {\it equational domain} ({\it e.d.} for shortness) in the language $\LL_S$ if for all algebraic sets $Y_1,Y_2,\ldots,Y_n$ the union $Y=Y_1\cup Y_2\cup\ldots\cup Y_n$ is algebraic. 

The next theorem contains necessary and sufficient conditions for a semigroup to be an e.d.  

\begin{theorem}\textup{\cite{uniTh_IV}}
\label{th:about_M}
A semigroup $S$ in the language $\LL_S$ is an e.d. iff the set 
\[
\M_{sem}=\{(x_1,x_2,x_3,x_4)|x_1=x_2\mbox{ or }x_3=x_4\}\subseteq S^4
\]
is algebraic, i.e. there exists a system $\Ss$ in variables $x_1,x_2,x_3,x_4$ with the solution set $\M_{sem}$.
\end{theorem}

Below we will study equations over groups, therefore we have to give some definitions of algebraic geometry over groups. Any group $G$ below will be considered in the language $\LL_G=\{\cdot,^{-1},1\}\cup\{g|g\in G\}$. \textit{An $\LL_G$-term}  in variables  $X=\{x_1,x_2,\ldots,x_n\}$ is a finite product of variables in integer powers and constants $g\in G$. In other words, an $\LL_G$-term is an element of the free product $F(X)\ast G$, where $F(X)$ is a free group generated by the set $X$. 

The definitions of equations, algebraic sets and equational domains over groups are given in the same way as it is over semigroups. 

For groups of the language $\LL_G$ we have the following result.

\begin{theorem}\textup{\cite{uniTh_IV}}
A group $G$ of the language $\LL_G$ is an e.d. iff the set 
\label{th:criterion_for_groups}
\begin{equation*}
\M_{gr}=\{(x_1,x_2)|x_1=1\mbox{ or }x_2=1\}\subseteq G^2
\end{equation*}
is algebraic, i.e. there exists a system $\Ss$ in variables $x_1,x_2$ with the solution set $\M_{gr}$.
\end{theorem}

One can reformulate Theorem~\ref{th:criterion_for_groups} in simpler form using the next definition. An element $x\neq 1$ of a group $G$ is a \textit{zero-divisor} if there exists $1\neq y\in G$ such that for any $g\in G$ it holds $[x,y^g]=1$ (here $y^g=gyg^{-1}$, $[a,b]=a^{-1}b^{-1}ab$).

\begin{theorem}\textup{\cite{uniTh_IV}}
\label{th:zero_divisors}
A group $G$ in the language $\LL_G$ is an e.d. iff it does not contain zero-divisors.
\end{theorem}

\bigskip

Let us give results of~\cite{shevl_ED_I}, where we studied equational domains in the class of finite simple semigroups.

The matrix $\P$ of a semigroup $S=(G,\P,\Lambda,I)$ is {\it non-singular} if it does not contain two equal rows or columns. The non-singularity of $\P$ is equivalent to the reductivity of the semigroup $S$. 

\begin{lemma}\textup{(Lemma~3.4. of~\cite{shevl_ED_I})}
\label{l:exists_2_non_dist_elems}
Suppose the matrix $\P$ of a finite simple semigroup $S=(G,\P,\Lambda,I)$ has equal rows (columns) with indexes $i,j$ (respectively, $\lambda,\mu$). Then for the elements $s_1=(1,1,i)$, $s_2=(1,1,j)$ (respectively, $s_1=(\lambda,1,1)$, $s_2=(\mu,1,1)$) and for an arbitrary $\LL_S$-term $t(x)$ one of the following conditions holds:
\begin{enumerate}
\item $t(s_1)=t(s_2)$;
\item $t(s_1)=(\nu,g,i)$, $t(s_2)=(\nu,g,j)$ for some $g\in G$, $\nu\in\Lambda$ if $t(x)$ ends with a variable $x$ (respectively, $t(s_1)=(\lambda,g,k)$, $t(s_2)=(\mu,g,k)$ for some $g\in G$, $k\in I$ if $t(x)$ begins with $x$). 
\end{enumerate}
\end{lemma}

\begin{theorem}\textup{(Theorem~3.1. of~\cite{shevl_ED_I})}
\label{th:main}
A finite completely simple semigroup $S=(G,\P,\Lambda,I)$ is an e.d. in the language $\LL_S$ iff the following two conditions hold:
\begin{enumerate}
\item $\P$ is nonsingular; 
\item $G$ is an e.d. in the group language $\LL_G$. 
\end{enumerate}
\end{theorem}

Remark that the ``if'' statement of Theorem~\ref{th:main} does not hold for infinite completely simple semigroups. 

\begin{example}\textup{(Example 4.11. of~\cite{shevl_ED_I})}
\label{ex:domain_240}
Define a finite simple semigroup $S_{240}=(A_5,\P,\{1,2\},\{1,2\})$, where $A_5$ is the alternating group of degree $5$, 
\[\P=\begin{pmatrix}1&1\\1&g\end{pmatrix},\]
and $g\neq 1$. By Theorem~\ref{th:main}, $S_{240}$ is an e.d. and $|S_{240}|=|A_5|\cdot 2\cdot 2=240$.
\end{example}

\begin{corollary}\textup{(Corollary~5.3. of~\cite{shevl_ED_I})}
\label{cor:about_homogroups}
If a homogroup $S$ is an e.d. then $S$ is a group, and $S=Ker(S)$.
\end{corollary}

\begin{corollary}\textup{(Corollary~5.3. of~\cite{shevl_ED_I})}
\label{cor:zero}
Any nontrivial semigroup $S$ with a zero is not an e.d. in the language $\LL_S$.
\end{corollary}

\section{Kernels of equational domains}
\label{sec:non_simple_semigroups}

It is easy to see that the set 
\begin{equation}
\label{eq:Gamma}
\Gamma=\{(1,g,1)|g\in G\}\subseteq K
\end{equation}
is isomorphic to $G$. Since $\P$ is normalized, $(1,1,1)$ is the identity of $\Gamma$.

Let
\[
L_i=\{(\lambda,g,i)|g\in G, \lambda\in\Lambda\}\subseteq K,
\]
\[
R_\lambda=\{(\lambda,g,i)|g\in G, i\in I\}\subseteq K.
\]
Obviously, $L_1\cap R_1=\Gamma$. By the properties of the kernel, any $L_i$ ($R_\lambda$) is a left (respectively, right) ideal of the semigroup $S$.

\begin{lemma}
\label{l:properties_of_multiplication}
Let $K=(G,\P,\Lambda,I)$ be a kernel of a semigroup $S$. Hence, for any $\alpha\in S$ there exist elements $g_\alpha\in G$, $\lambda_\alpha\in\Lambda$, $i_\alpha\in I$ such that 
\begin{enumerate}
\item $\alpha(1,1,1)=(\lambda_\alpha,g_\alpha,1)$,
\item $\alpha(1,g,i)=(\lambda_\alpha,g_\alpha g,i)$,
\item $(1,1,1)\alpha=(1,g_\alpha,i_\alpha)$,
\item $(\lambda,g,1)\alpha=(\lambda,gg_\alpha,i_\al)$,
\end{enumerate}
\end{lemma}
\begin{proof}
Let us prove all statements of the lemma.
\begin{enumerate}
\item Since the set $L_1$ is a left ideal, the element $\alpha(1,1,1)$ equals to the expression $(\lambda_\alpha,g_\alpha,1)$ for some $\lambda_\al,g_\al$.
\item
\[
\al(1,g,i)=\alpha(1,1,1)(1,g,i)=
(\lambda_\alpha,g_\alpha,1)(1,g,i)=(\lambda_\alpha,g_\alpha g,i).
\]
\item Since $R_1$ is a right ideal, $(1,1,1)\alpha=(1,h_\alpha,i_\alpha)$ for some $h_\alpha,i_\alpha$. Let us prove that $h_\al=g_\al$:
\[
(1,1,1)(\al(1,1,1))=(1,1,1)(\lambda_\al,g_\al,1)=(1,g_\al,1).
\]
On the other hand, computing
\[
((1,1,1)\al)(1,1,1)=(1,h_\alpha,i_\alpha)(1,1,1)=(1,h_\al,1),
\]
we obtain $h_\al=g_\al$.

\item 
\[
(\lambda,g,1)\al=(\lambda,g,1)(1,1,1)\al=(\lambda,g,1)(1,g_\alpha,i_\alpha)=(\lambda,gg_\al,i_\al).
\]

\end{enumerate}
\end{proof}

According to Lemma~\ref{l:properties_of_multiplication} for any $\al\in S$ we have:
\begin{equation}
\label{eq:relation_for_Lambda1}
\alpha x=(\lambda_\alpha,g_\alpha,1)x\mbox{ for each }x\in L_1
\end{equation} 
\begin{equation}
\label{eq:relation_for_I1}
x\alpha=x(1,g_\alpha,i_\alpha)\mbox{ for each }x\in R_1
\end{equation} 
\begin{equation}
\label{eq:relation_for_Gamma}
\alpha x=(\lambda_\alpha,g_\alpha,1)x,\;x\alpha=x(1,g_\alpha,i_\alpha)\mbox{ for each }x\in \Gamma
\end{equation}

\begin{lemma}\textup{(Lemma~3.6. of~\cite{shevl_ED_I})}
\label{l:about_equiv_over_Gamma}
Let $S=(G,\P.\Lambda,I)$ be a completely simple semigroup, and $x,y\in\Gamma$. Then
\begin{enumerate}
\item $x(\lambda,c,i)y=x(1,c,1)y$;
\item if an equation 
\begin{equation*}
(\lambda,c,i)t(x,y)=(\lambda^\pr,c^\pr,i^\pr)t^\pr(x,y) \; (respectively, t(x,y)(\lambda,c,i)=t^\pr(x,y)(\lambda^\pr,c^\pr,i^\pr)
)
\end{equation*}
is consistent over $S$, then it is equivalent to   
\[
(1,c,1)t(x,y)=(1,c^\pr,1)t^\pr(x,y)\; (\mbox{respectively, } t(x,y)(1,c,1)=t^\pr(x,y)(1,c^\pr,1))
\]
over the group $\Gamma$.
\end{enumerate}
\end{lemma}

\begin{lemma}
\label{l:S-eq_dom->G-eq_dom_new}
Suppose a semigroup $S$ with the kernel $K=(G,\P,\Lambda,I)$ is an equational domain in the language $\LL_S$. Then the group $G$ is an e.d. in the language $\LL_G$. 
\end{lemma}
\begin{proof}
As $S$ is an e.d., the set of pairs $\M=\{(x,y)|x=(1,1,1)\mbox{ or }y=(1,1,1)\}\subseteq S^2$ is algebraic over $S$, i.e. there exists a system $\Ss(x,y)$ over $\LL_S$ such that $\V_S(\Ss)=\M$.

Below we rewrite $\Ss(x,y)$ into a system with constants from $\Gamma$.

Let us show that $\Ss$ does not contain any equation of the form $t(x,y)=\al$, where $\al\in S\setminus K$. Indeed, the value of a term $t(x,y)$ at $((1,1,1),(1,1,1))\in\M$ belongs to the kernel $K$, and the equality $t((1,1,1),(1,1,1))=\al$ is impossible.

Thus, all parts of equations of $\Ss$ contain occurrences of variables. Applying Lemma~\ref{l:about_equiv_over_Gamma}, we obtain that $\Ss$ is equivalent over  $\Gamma$ to a system $\tilde{\Ss}$ whose constants belong to $K$. 

If an equation of the form 
\begin{equation}
(\lambda,g,i)t^\pr(x,y)=xs^\pr(x,y)
\label{eq:cdssadfsa}
\end{equation}
belongs to $\tilde{\Ss}$, it is not satisfied by the point $((\mu,h,j),(1,1,1))\in\M$, where $\mu\neq\lambda$. Thus, $\tilde{\Ss}$ does not contain equations of the form~(\ref{eq:cdssadfsa}). It allows us to apply the formulas from Lemma~\ref{l:about_equiv_over_Gamma} to $\tilde{\Ss}$ and obtain a system $\Ss^\pr$ whose constants belong to the group $\Gamma$. Moreover, $\Ss^\pr$ is equivalent to $\Ss$ over the group $\Gamma$.

Finally, we have $\V_\Gamma(\Ss^\pr)=\{(x,y)|x=(1,1,1)\mbox{ or }y=(1,1,1)\}\subseteq\Gamma^2$, and, by Theorem~\ref{th:criterion_for_groups}, the group $\Gamma$ is an equational domain in the language  $\LL_\Gamma$. The isomorphism between the groups $\Gamma,G$ proves the lemma.

\end{proof}

\begin{lemma}
\label{l:singular->not_ED_new}
If a semigroup $S$ with the kernel $K=(G,\P,\Lambda,I)$ is an e.d. in the language $\LL_S$, then the matrix $\P$ is nonsingular. 
\end{lemma}
\begin{proof}
Assume that $S$ is an e.d. with a singular matrix $\P$, and the $i$-th, $j$-th rows of $\P$ are equal (similarly, one can consider a matrix $\P$  with two equal columns). 

Since the semigroup $S$ is an e.d., there exists a system of equations $\Ss(x,y,z)$ with the solution set
\[
\M=\{(x,y,z)|x=(1,1,i)\mbox{ or }y=(1,1,i)\mbox{ or }z=(1,1,i)\}.
\]
Let $t(x,y,z)=s(x,y,z)\in\Ss(x,y,z)$ such that $t(x,y,z)=s(x,y,z)$ is not satisfied by the point $Q=((1,1,j),(1,1,j),(1,1,j))$.

By the formula~(\ref{eq:relation_for_I1}), $t(x,y,z)=s(x,y,z)$ is equivalent over the subsemigroup $R_1$ to the one of the following equations:
\begin{enumerate}
\item $t^\pr(x,y,z)=s^\pr(x,y,z)$;
\item $\alpha t^\pr(x,y,z)=\beta s^\pr(x,y,z)$;
\item $\al t^\pr(x,y,z)=s^\pr(x,y,z)$;
\item $t^\pr(x,y,z)=\beta s^\pr(x,y,z)$;
\item $\al t^\pr(x,y,z)=\beta$;
\item $t^\pr(x,y,z)=\beta$,
\end{enumerate} 
where all constants of the terms $t^\pr(x,y,z),s^\pr(x,y,z)$ belong to the kernel $K$, and $\al,\beta\in S\setminus K$.

Consider only the second type of the equation $t(x,y,z)=s(x,y,z)$ (similarly, one can consider the other types).

Without loss of generality we can assume that  neither $t^\pr(x,y,z)$ nor $s^\pr(x,y,z)$ ends by $z$.

Consider the following terms in one variable $t^{\pr\pr}(z)=t^\pr((1,1,j),(1,1,j),z)$, $s^{\pr\pr}(z)=s^\pr((1,1,j),(1,1,j),z)$. Each constant of the terms $t^{\pr\pr}(z)$, $s^{\pr\pr}(z)$ belong to the kernel $K$, and the terms end with constants. By Lemma~\ref{l:exists_2_non_dist_elems} we have the equalities
\[
t^{\pr\pr}((1,1,i))=t^{\pr\pr}((1,1,j)),\; s^{\pr\pr}((1,1,i))=s^{\pr\pr}((1,1,j)).
\]

Since $((1,1,j),(1,1,j),(1,1,i))\in\M$, we have 
\[
\al t^\pr((1,1,j),(1,1,j),(1,1,i))=\beta s^\pr((1,1,j),(1,1,j),(1,1,i)).
\]
Using the equalities above, we obtain
\[
\al t^\pr((1,1,j),(1,1,j),(1,1,j))=\beta s^\pr((1,1,j),(1,1,j),(1,1,j)),
\]
which contradicts the choice of the equation $t(x,y,z)=s(x,y,z)$.

\end{proof}

\begin{theorem}
\label{th:main_new}
Suppose a semigroup $S$ has the finite kernel $K=(G,\P,\Lambda,I)$, and $S$ is an e.d. in the language $\LL_S$. Then $K$ is an e.d. in the language $\LL_K$.
\end{theorem}
\begin{proof}
By Lemma~\ref{l:S-eq_dom->G-eq_dom_new}, the group $G$ is an e.d. in the language $\LL_G$. Lemma~\ref{l:singular->not_ED_new} gives us the non-singularity of the matrix $\P$. Finally, Theorem~\ref{th:main} concludes the proof. 
\end{proof}

A semigroup $S$ is called a {\it proper e.d.}, if $S$ is an e.d. in the language $\LL_S$ and $S$ is not a group. 

\begin{corollary}
\label{cor:from_main_new}
If $S$ is a proper e.d. then  $|S|\geq 240$.
\end{corollary}
\begin{proof}
Let $K=(G,\P,\Lambda,I)$ be the kernel of the semigroup $S$. By Theorem~\ref{th:main_new}, we have that 
\begin{enumerate}
\item the matrix $\P$ is nonsingular. It means that either $\P$ has at least two rows and columns or  $|\Lambda|=|I|=1$;
\item the group $G$ is an e.d. in the language $\LL_G$.
\end{enumerate}

If $|\Lambda|=|I|=1$ the kernel is a group, therefore $S$ becomes a homogroup. By Corollary~\ref{cor:about_homogroups}, we obtain $S=K$, and $S$ is not a proper e.d.

Thus, the matrix $\P$ contains at least two rows and columns. Therefore, the group $G$ is not trivial
(the triviality of $G$ makes $\P$ singular). 

Since the alternating group $A_5$ is a nontrivial e.d. of the minimal order, we obtain the next estimation:
\[|S|\geq|K|=|G||\Lambda||I|\geq|A_5||\Lambda||I|\geq 60\cdot 2\cdot 2=240.\]

Recall that we cannot improve the estimation above (see Example~\ref{ex:domain_240}).
\end{proof}

\section{Inner translations of ideals}
\label{sec:action_on_ideal}

Suppose a semigroup $S$ has an ideal $I$ and $\al\in S$. A map $l_\al(x)\colon I\to I$ ($r_\al(x)\colon I\to I$) is called a \textit{left (respectively, right) inner translation of the ideal $I$} if $l_\al(x)=\alpha x$ (respectively, $r_\al(x)=x\alpha$) for $x\in I$. Let $\al\sim_I \beta$ denote elements $\al,\beta\in S$ such that
\[
\al x=\beta x, \; \mbox{ and } x\al=x\beta
\]
for any $x\in I$. If $\al\sim_I \beta$ the elements $\al,\beta$ are called $I$-{\it equivalent}. The equivalence relation $\sim_I$ is \textit{trivial} if any equivalence class consists of a single element, i.e. a trivial equivalence relation $\sim_I$ coincides with the equality relation over a semigroup $S$.

\begin{remark}
Notice that the triviality of $\sim_I$ relation is close to the definition of a weakly reductive semigroup. Namely, a semigroup $S$ is weakly reductive iff the relation $\sim_S$ is trivial.
\end{remark}

\begin{lemma}
\label{l:zamena}
Suppose a semigroup $S$ has an ideal $I$ and $\al\sim_I\beta$. Then for any term $t(x,y)$ containing occurrences of the variable $y$ we have
\begin{equation}
t(\al,r)=t(\beta,r)
\label{eq:t(al,r)=t(beta,r)}
\end{equation} 
for all $r\in I$.
\end{lemma}
\begin{proof}
Let $t(x,y)$ be a term of the language $\LL_S$. The {\it length} $|t(x,y)|$ of a term $t(x,y)$ is the length of the word $t(x,y)$ is the alphabet $X\cup\{s|s\in S\}$. For example, $|xs_1y^2|=4$, $|x|=|s_2|=1$.

\begin{enumerate}
\item Let $t(x,y)=v(x)y^n$. We prove the statement of the lemma by induction on the length of $v(x)$. 

If $|v(x)|=1$, then $v(x)$ is either a constant $\c$ or $v(x)=x$. If $v(x)=\c$ the equality~(\ref{eq:t(al,r)=t(beta,r)}) obviously holds. If $v(x)=x$, by the condition of the lemma, we have $\al r=\beta r$ and obtain~(\ref{eq:t(al,r)=t(beta,r)}).

Assume that~(\ref{eq:t(al,r)=t(beta,r)}) holds for any term of the length less than $n$. Let $|v(x)|=n$. We have either $v(x)=v^\pr(x)\c$ or $v(x)=v^\pr(x)x$, where $|v^\pr(x)|=n-1$. 

If $v(x)=v^\pr(x)\c$, then we use $\c r^n\in I$ and apply the induction hypothesis:
\[
t(\al,r)=v^\pr(\al)(\c r^n)=v^\pr(\beta)(\c r^n)=t(\beta,r).
\]
For $v(x)=v^\pr(x)x$ the proof is similar.

\item If $t(x,y)=y^n v(x)$ the proof is similar to the reasonings of the previous case.

\item Consider the most general form of the term  $t(x,y)$:
\[
t(x,y)=w_1(x)y^{n_1}w_2(x)y^{n_2}\ldots w_m(x)y^{n_m}w_{m+1}(x),
\]
where the terms $w_1(x),w_{m+1}(x)$ may be empty.

Let us prove the statement of the lemma by the induction on $m$. If $m=1$, the term $t(x,y)$ is reduced to the terms from the previous cases.

Assume $t^\pr(\al,r)=t^\pr(\beta,r)=r^\pr\in I$, where
\[
t^\pr(x,y)=w_1(x)y^{n_1}w_2(x)y^{n_2}\ldots w_{m-1}(x)y^{n_{m-1}}w_m(x).
\]

Let us consider the term $s(x,y)=y^{n_m} w_{m+1}(x)$. As we proved above,
\[
s(\al,r)=s(\beta,r).
\]
Thus,
\[
t(\al,r)=r^{\pr}s(\al,r)=r^{\pr}s(\beta,r)=t(\beta,r),
\]
which proves~(\ref{eq:t(al,r)=t(beta,r)}). 

\end{enumerate}

\end{proof}

Notice that the statement of Lemma~\ref{l:zamena} fails for terms $t(x,y)$ which do not depend on the variable $y$. 

\begin{theorem}
\label{th:alpha_sim_beta}
If a semigroup $S$ is an e.d. in the language $\LL_S$, the equivalence relation $\sim_I$ is trivial for any ideal $I\subseteq S$. 
\end{theorem}
\begin{proof}
Since the theorem is obviously holds for trivial semigroup, further we put $|S|>1$.

Let $X$ denote the set of four variables $\{x_1,x_2,x_3,x_4\}$.

If $I=\{r\}$ the element $r$ is the zero of $S$. By Corollary~\ref{cor:zero}, $S$ is not an e.d. Let $r_1,r_2$ be two distinct elements of the ideal $I$.

Assume the converse: there exist distinct elements $\alpha,\beta\in S$ with $\al\sim_I\beta$.

By the condition, the set $\M_{sem}=\{(x_1,x_2,x_3,x_4)|x_1=x_2\mbox{ or }x_3=x_4\}$ is the solution set of a system $\Ss(X)$, and an equation $t(X)=s(X)\in\Ss$ is not satisfied by the point $(\al,\beta,r_1,r_2)\notin \M_{sem}$.

Let $\Var(t)$ be the set of variabless occurring in a term $t$.

For the equation $t(X)=s(X)$ we have one of the following conditions:
\begin{enumerate}
\item a variable $x_i\in X$ does not occur in the equation;
\item all intersections $\Var(t)\cap\{x_1,x_2\}$, $\Var(s)\cap\{x_1,x_2\}$, $\Var(t)\cap\{x_3,x_4\}$, $\Var(s)\cap\{x_3,x_4\}$ are nonempty;
\item $\Var(t)\cap\{x_1,x_2\}\neq\emptyset$, $\Var(s)\cap\{x_1,x_2\}\neq\emptyset$, $\Var(t)\cap\{x_3,x_4\}=\emptyset$ , $\Var(s)\cap\{x_3,x_4\}\neq\emptyset$;
\item $\Var(t)\cap\{x_3,x_4\}\neq\emptyset$, $\Var(s)\cap\{x_3,x_4\}\neq\emptyset$, $\Var(t)\cap\{x_1,x_2\}\neq\emptyset$ , $\Var(s)\cap\{x_1,x_2\}=\emptyset$;
\item one part of the equation contains only the variables $\{x_1,x_2\}$, and the another part contains only $\{x_3,x_4\}$;
\item one of the parts of the equation does not contain any variable.
\end{enumerate}

Let us consider all types of $t(X)=s(X)$.
\begin{enumerate}
\item Without loss of generality one can assume that $t(X)=s(X)$ does not contain the occurrences of $x_4$, i.e. $t(x_1,x_2,x_3)=s(x_1,x_2,x_3)$. Since $(\al,\beta,r_1,r_1)\in\M_{sem}$, we have $t(\al,\beta,r_1)=s(\al,\beta,r_1)$. However the point $(\al,\beta,r_1,r_2)$ does not satisfy the equation $t(X)=s(X)$, and we obtain the contradiction $t(\al,\beta,r_1)\neq s(\al,\beta,r_1)$.

\item By Lemma~\ref{l:zamena} we have:
\[
t(\al,\al,r_1,r_2)=t(\al,\beta,r_1,r_2),\; s(\al,\al,r_1,r_2)=s(\al,\beta,r_1,r_2).
\] 
Since $(\al,\al,r_1,r_2)\in\M_{sem}$, we have $t(\al,\al,r_1,r_2)=s(\al,\al,r_1,r_2)$. Therefore,
\[
t(\al,\beta,r_1,r_2)=s(\al,\beta,r_1,r_2),
\]
which contradicts the choice of the equation $t(X)=s(X)$.
\item Let $t(X)=s(X)$ be $t(x_1,x_2)=s(x_1,x_2,x_3,x_4)$.
By Lemma~\ref{l:zamena}, we have:
\[
s(\al,\al,r_1,r_1)=s(\al,\beta,r_1,r_1).
\] 
Since $(\al,\al,r_1,r_1),(\al,\beta,r_1,r_1)\in\M_{sem}$, we obtain 
\[
t(\al,\al)=s(\al,\al,r_1,r_1),\; t(\al,\beta)=s(\al,\beta,r_1,r_1),
\]
and $t(\al,\al)=t(\al,\beta)$.

Since $(\al,\al,r_1,r_2)\in\M_{sem}$, Lemma~\ref{l:zamena} gives us the equalities
\[
t(\al,\al)=s(\al,\al,r_1,r_2)=s(\al,\beta,r_1,r_2).
\]
Therefore,
\[
t(\al,\beta)=s(\al,\beta,r_1,r_2),
\]
which contradicts the choice of the equation $t(X)=s(X)$.

\item Suppose the equation $t(X)=s(X)$ is $t(x_1,x_2,x_3,x_4)=s(x_3,x_4)$. For the point  $(\al,\al,r_1,r_2)\in\M_{sem}$ we have
\[t(\al,\al,r_1,r_2)=s(r_1,r_2).\]
By Lemma~\ref{l:zamena}, we obtain
\[t(\al,\al,r_1,r_2)=t(\al,\beta,r_1,r_2),\]
therefore,
\[t(\al,\beta,r_1,r_2)=s(r_1,r_2),\]
which contradicts the choice of the the equation $t(X)=s(X)$.

\item Let the equation $t(X)=s(X)$ be $t(x_1,x_2)=s(x_3,x_4)$. For the points $(\al,\al,r_1,r_1),(\al,\beta,r_1,r_1),(\al,\al,r_1,r_2)\in\M_{sem}$ we have the equalities 
\[
t(\al,\al)=s(r_1,r_1),\; t(\al,\beta)=s(r_1,r_1),\;t(\al,\al)=s(r_1,r_2),
\] 
therefore, $t(\al,\beta)=s(r_1,r_2)$, which contradicts the choice of the equation $t(X)=s(X)$.

\item Assume that $t(X)=s(X)$ is $t(x_1,x_2,x_3,x_4)=\c$, where $\c\in S$.
According to Lemma~\ref{l:zamena} the equality $t(\al,\al,r_1,r_2)=\c$ implies $t(\al,\beta,r_1,r_2)=\c$. The last equality contradicts the choice of the equation $t(X)=s(X)$. 

\end{enumerate}

\end{proof}

\begin{corollary}
\label{cor:S>l^2l}
Let $I$ be a finite ideal of a semigroup $S$, $|I|=l$. If $|S|>l^{2l}$ the semigroup $S$ is not an e.d. in the language $\LL_S$.
\end{corollary}
\begin{proof}
Let us show that the condition $|S|>l^{2l}$ implies the existence of two distinct elements $\al,\beta$ with $\al\sim_I\beta$.

Indeed, any left inner translation $l_\al$ of the ideal $I$ is a function $l_\al\colon I\to I$. The number of different mappings over the set $I$ equals $l^l$. Similarly, the number of different right inner translations of $I$ is also equal to $l^l$. Hence, the number of $\sim_I$-classes is not more than  $l^l\cdot l^l=l^{2l}$. 

Thus, the equality $|S|>l^{2l}$ implies the existence of two distinct elements $\al,\beta$ with $\al\sim_I\beta$. By Theorem~\ref{th:alpha_sim_beta}, $S$ is not an e.d. in the language $\LL_S$.
\end{proof}

\begin{corollary}
\label{cor:about_infinite_semigroups}
Any infinite semigroup $S$ with a finite ideal $I$ is not an e.d. in the language $\LL_S$.
\end{corollary}

\bigskip

Let us improve the estimation of Corollary~\ref{cor:S>l^2l} for a semigroup $S$ with the finite kernel $K$.

\begin{lemma}
\label{l:about_action}
Suppose a semigroup $S$ has the finite kernel $K=(G,\P,\Lambda,I)$. Then, for any $\alpha\in S$ there exist an element $g_\alpha\in G$ and mappings $\Lambda_\alpha\colon\Lambda\to\Lambda$, $I_\al\colon I\to I$ such that
\begin{enumerate}
\item $\alpha(\lambda,1,1)=(\Lambda_\al(\lambda),g_\al p_{I_\al(1)\lambda},1)$,
\item $(1,1,i)\alpha=(1,p_{i\Lambda_\al(1)}g_\al,I_\al(i))$,
\item $\al(\lambda,g,i)=(\Lambda_\al(\lambda),g_\al p_{I_\al(1)\lambda}g,i)$,
\item $(\lambda,g,i)\al=(\lambda,gp_{i\Lambda_\al(1)}g_\al,I_\al(i))$.
\end{enumerate}
\end{lemma}
\begin{proof}
Clearly, $I_\al(1)=i_\al$, $\Lambda_\al(1)=\lambda_\al$, where the indexes $i_\al,\lambda_\al$ are defined in Lemma~\ref{l:properties_of_multiplication}.

Let us prove the statements of the lemma.

\begin{enumerate}
\item Since $L_1$ is a left ideal, $\al(\lambda,1,1)\in L_1$. Then $\al(\lambda,1,1)=(\Lambda_\al(\lambda),G_\al(\lambda),1)$, where $G_\al(\lambda)\colon\Lambda\to G$ is a map depending on $\lambda$.

We have
\[
(1,1,1)(\al(\lambda,1,1))=(1,1,1)(\Lambda_\al(\lambda),G_\al(\lambda),1)=(1,G_\al(\lambda),1).
\] 
Using Lemma~\ref{l:properties_of_multiplication},
\[
((1,1,1)\al)(\lambda,1,1)=(1,g_\al,i_\al)(\lambda,1,1)=(1,g_\al p_{i_\al\lambda},1)=(1,g_\al p_{I_\al(1)\lambda},1),
\] 
we obtain $G_\al(\lambda)=g_\al p_{I_\al(1)\lambda}$.

\item Since $R_1$ is a right ideal, $(1,1,i)\al\in R_1$. Then $(1,1,i)\al=(1,G_\al^\pr(i),I_\al(i))$, where $G_\al^\pr\colon I\to G$ is map depending on $i$.

We have
\[
((1,1,i)\al)(1,1,1)=(1,G_\al^\pr(i),I_\al(i))(1,1,1)=(1,G_\al(i)^\pr,1).
\] 
On the other hand,
\[
(1,1,i)(\al(1,1,1))=(1,1,i)(\lambda_\al,g_\al,1)=(1,p_{i\lambda_\al}g_\al,1)=(1,p_{i\Lambda_\al(1)}g_\al,1),
\] 
thus $G_\al^\pr(i)=g_\al p_{i\Lambda_\al(1)}$.
\item
\[
\al(\lambda,g,i)=\al(\lambda,1,1)(1,g,i)=(\Lambda_\al(\lambda),g_\al p_{I_\al(1)\lambda},1)(1,g,i)=(\Lambda_\al(\lambda),g_\al p_{I_\al(1)\lambda}g,i),
\]

\item
\[
(\lambda,g,i)\al=(\lambda,g,1)(1,1,i)\al=(\lambda,g,1)(1,p_{i\Lambda_\al(1)}g_\al,I_\al(i))=(\lambda,gp_{i\Lambda_\al(1)}g_\al,I_\al(i)).
\]

\end{enumerate}
\end{proof}

\begin{theorem}
Suppose a semigroup $S$ has the finite kernel $K=(G,\P,\Lambda,I)$. If $|S|>|G||\Lambda|^{|\Lambda|}|I|^{|I|}$ the semigroup $S$ is not an e.d. in the language $\LL_S$.
\end{theorem}
\begin{proof}
According to Lemma~\ref{l:about_action}, any inner translations of the kernel $K$ is defined by an element $g_\al\in G$ and mappings $\Lambda_\alpha\colon\Lambda\to\Lambda$, $I_\al\colon I\to I$. Therefore, there do not exist more than $|G||\Lambda|^{|\Lambda|}|I|^{|I|}$ different inner translations of the kernel.

Thus, the inequality $|S|>|G||\Lambda|^{|\Lambda|}|I|^{|I|}$ implies an existence of two different elements $\al,\beta\in S$ with the same translation of the kernel. By Theorem~\ref{th:alpha_sim_beta}, we obtain the statement of the theorem.
\end{proof}

\begin{example}
Let $S_{240}$ be a finite simple semigroup defined in Example~\ref{ex:domain_240}. Therefore, any equational domain with the kernel isomorphic to $S_{240}$ contains at most $60\cdot 2^2\cdot 2^2=960$ elements.  
\end{example}

\section{Semigroups with finite ideals}
\label{sec:criterion}

Let $K=(G,\P,\Lambda,I)$ be the finite kernel of a semigroup $S$ and $M\subseteq S^n$. By $\T(M,\Gamma)$ denote the set of all terms of the language $\LL_S$ in variables $x_1,x_2,\ldots,x_n$ such that
\begin{enumerate}
\item all constants of a term $t(X)\in\T(M,\Gamma)$ belong to the kernel $K$;
\item the value $t(P)$ of $t(X)\in\T(M,\Gamma)$ belong to the subgroup $\Gamma$ defined by formula~(\ref{eq:Gamma}) for all $P\in M$.
\end{enumerate}
For example,
\[
t(x)=(1,g,2)x(3,h,1)\in \T(S,\Gamma),
\]
\[
s(x,y)=(1,g,1)x^2(3,h,4)y(2,f,1)\T(S^2,\Gamma).
\]

\begin{lemma}
\label{l:exists_dist_term_new}
Suppose a semigroup $S$ has the finite kernel $K=(G,\P,\Lambda,I)$, where the matrix $\P$ is nonsingular and the equivalence relation $\sim_K$ is trivial. Then for any pair of distinct elements $\al,\beta\in S$ there exists a term $t(x)\in\T(S,\Gamma)$ with $t(\al)\neq t(\beta)$. 
\end{lemma}
\begin{proof}
Since $\sim_K$ is trivial, there exists an element $(\lambda,g,i)\in K$ with $\al(\lambda,g,i)\neq \beta(\lambda,g,i)$. According to Lemma~\ref{l:about_action}, we have the equalities
\[
\al(\lambda,g,i)=(\Lambda_\al(\lambda),g_\al g,i),
\]
\[
\beta(\lambda,g,i)=(\Lambda_\beta(\lambda),g_\beta g,i).
\]
There are exactly two possibilities.
\begin{enumerate}
\item Let $g_\al\neq g_\beta$. Consider the term $t(x)=(1,1,1)x(\lambda,1,1)\in\T(S,\Gamma)$. We have
\[
t(\al)=(1,1,1)\al(\lambda,1,1)=(1,1,1)(\Lambda_{\al}(\lambda),g_\al,1)=(1,g_\al,1),
\]
\[
t(\beta)=(1,1,1)\beta(\lambda,1,1)=(1,1,1)(\Lambda_{\beta}(\lambda),g_\beta,1)=(1,g_\beta,1),
\]
thus $t(\al)\neq t(\beta)$.
\item Suppose $g_\al=g_\beta=g$ and $\Lambda_\al(\lambda)\neq\Lambda_\beta(\lambda)$. Since the matrix $\P$ is nonsingular, there exists an index $i$ with $p_{i\Lambda_\al(\lambda)}\neq p_{i\Lambda_\beta(\lambda)}$. Consider the term $t(x)=(1,1,i)x(\lambda,1,1)\in\T(S,\Gamma)$. We have
\[
t(\al)=(1,1,i)\al(\lambda,1,1)=(1,1,i)(\Lambda_{\al}(\lambda),g,1)=(1,p_{i\Lambda_\al(\lambda)}g,1),
\]
\[
t(\beta)=(1,1,i)\beta(\lambda,1,1)=(1,1,i)(\Lambda_{\beta}(\lambda),g,1)=(1,p_{i\Lambda_\beta(\lambda)}g,1),
\]
thus $t(\al)\neq t(\beta)$.
\end{enumerate}
\end{proof}

Let $P=(p_1,p_2,\ldots,p_n)\in S^n$. By $\T_P(M,\Gamma)$ (where $P\in M\subseteq S^n$) we denote the set of all terms $t(X)\in\T(S^n,\Gamma)$ such that $t(P)\neq(1,1,1)$, and $t(Q)=(1,1,1)$ for any $Q\in M\setminus\{P\}$.

\begin{lemma}
\label{l:sufficient_conditions_new}
Suppose a finite semigroup $S$ has the kernel $K=(G,\P,\Lambda,I)$, the equivalence relation $\sim_K$ is trivial and the kernel $K$ is an e.d. in the language $\LL_K$. Then for any natural $n$ and an arbitrary point $P=(p_1,p_2,\ldots,p_n)\in S^n$ the set $\T_P(S^n,\Gamma)$ 
\begin{enumerate}
\item is nonempty;
\item for any term $t(X)\in\T_P(S^n,\Gamma)$ and any $g\in G$ it holds $(1,g,1)t(X)(1,g^{-1},1)\in \T_P(S^n,\Gamma)$.
\end{enumerate}
\end{lemma}
\begin{proof}
The second property of the set $\T_P(S^n,\Gamma)$ easily follows from the first one. Indeed, if $t(X)\in\T_P(S^n,\Gamma)$, then
\[
(1,g,1)t(Q)(1,g^{-1},1)=(1,g,1)(1,1,1)(1,g^{-1},1)=(1,g1g^{-1},1)=(1,1,1),
\]
\[
(1,g,1)t(P)(1,g^{-1},1)=(1,g,1)(1,h,1)(1,g^{-1},1)=(1,ghg^{-1},1)\neq(1,1,1),\mbox{ since }h\neq 1.
\]

Let us prove now $\T_P(S^n,\Gamma)\neq\emptyset$.

Below we shall use the denotation:
\[t^{-1}(X)=t^{|G|-1}(X).\]
Obviously, for any $t(X)\in\T(S^n,\Gamma)$ it holds $t^{-1}(X)\in\T(S^n,\Gamma)$, and
\[
t(X)t^{-1}(X)=t^{-1}(X)t(X)=t^{|G|}(X)=(1,1,1)\mbox{ for all }X\in S^n.
\]

We prove $\T_P(M,\Gamma)\neq\emptyset$ by induction on the cardinality of the set $M\subseteq S^n$. Let $M=\{P,Q\}\subseteq S^n$.

Without loss of generality one can assume that the points $P,Q$ have distinct first coordinates $p_1\neq q_1$. By Lemma~\ref{l:exists_dist_term_new}, there exists a term $t(x)\in\T(S,\Gamma)$ with $t(p_1)\neq t(q_1)$. Let $s(X)=t(x_1)t^{-1}(q_1)\in\T(S,\Gamma)$, and we have $s(P)=t(p_1)t^{-1}(q_1)\neq (1,1,1)$, $s(Q)=t(q_1)t^{-1}(q_1)=(1,1,1)$. Thus, $s(X)\in \T_P(M,\Gamma)$.  

Suppose that for any set $M$ with $|M|\leq m$ the statement of the lemma is proved. Let us prove the lemma for a set $M$ with $m+1$ elements.

Let $M=\{P,Q_1,Q_2,\ldots,Q_m\}$. By the assumption of induction, there exist terms
\[t(X)\in\T_P(\{P,Q_2,Q_3,\ldots,Q_m\},\Gamma),s(X)\in\T_P(\{P,Q_1,Q_3,\ldots,Q_m\},\Gamma),\]
with values

\begin{tabular}{ccccccc}
&$P$&$Q_1$&$Q_2$&$Q_3$&$\ldots$&$Q_m$\\
$t(X)$&$(1,g_1,1)$&$(1,h_1,1)$&$(1,1,1)$&$(1,1,1)$&$\ldots$&$(1,1,1)$\\
$s(X)$&$(1,g_2,1)$&$(1,1,1)$&$(1,h_2,1)$&$(1,1,1)$&$\ldots$&$(1,1,1)$
\end{tabular} 

One can choose the elements $g_1,g_2\in G$ which do not commute. Indeed, the second property of the set $\T_P(M,\Gamma)$ allows us to take $g_2$ from the conjugacy class $C=\{gg_2g^{-1}|g\in G\}$. If $g_1$ commutes with all elements of $C$ then $g_1$ is a zero-divisor in the group $G$, and, by Theorem~\ref{th:zero_divisors}, $G$ is not an e.d.   

The values of the term
\[
p(X)=t^{-1}(X)s^{-1}(X)t(X)s(X)\in\T(S,\Gamma),
\]
are

\begin{tabular}{ccccccc}
&$P$&$Q_1$&$Q_2$&$Q_3$&$\ldots$&$Q_m$\\
$p(X)$&$(1,[g_1,g_2],1)$&$(1,1,1)$&$(1,1,1)$&$(1,1,1)$&$\ldots$&$(1,1,1)$
\end{tabular} 

where $[g_1,g_2]=g_1^{-1}g_2^{-1}g_1g_2\neq 1$ is the commutator of  $g_1,g_2$ in the group $G$.
 
Thus, $p(X)\in\T_P(M,\Gamma)$, and we prove the lemma.

\end{proof}

\begin{theorem}
\label{th:criterion}
A semigroup $S$ with the finite kernel $K=(G,\P,\Lambda,I)$ is an e.d. in the language $\LL_S$ iff the next conditions holds:
\begin{enumerate}
\item the kernel $K$ is an e.d. in the language $\LL_K$;
\item the equivalence relation $\sim_K$ is trivial.
\end{enumerate}
\end{theorem}
\begin{proof}
The ``only if'' statement follows from Theorems~\ref{th:main_new},~\ref{th:alpha_sim_beta}. 

Let us prove the ``if'' statement of the theorem.

Prove the converse. If $S$ is infinite, the equivalence relation $\sim_K$ for a finite $K$ is nontrivial, and we came to the contradiction. Thus, the semigroup $S$ is finite. 

Consider a set $\M_{sem}=\{(x_1,x_2,x_3,x_4)|x_1=x_2\mbox{ or }x_3=x_4\}\subseteq S^4$. By Lemma~\ref{l:sufficient_conditions_new} for any point $P\notin \M_{sem}$ there exists a term $t_P(x_1,x_2,x_3,x_4)\in\T_P(S^4,\Gamma)$. Obviously, the solution set of the system $\Ss=\{t_P(x_1,x_2,x_3,x_4)=(1,1,1)|P\notin \M_{sem}\}$ equals $\M_{sem}$. Thus, the set $\M_{sem}$ is algebraic and by Theorem~\ref{th:about_M} the semigroup $S$ is an e.d.

\end{proof}

\begin{corollary}
Let a semigroup $S$ with the finite kernel $K$ be an e.d. in the language $\LL_S$. Then any nonempty set $M\subseteq S^n$ is defined by a system of the form $\Ss=\{t_i(X)=(1,1,1)|i\in \mathcal{I}\}$, where $t_i(X)\in\T(S^n,\Gamma)$.  
\end{corollary}
\begin{proof}
By Corollary~\ref{cor:about_infinite_semigroups}, the semigroup $S$ is finite.

Let $S^n\setminus M=\{P_i|i\in\mathcal{I}\}$. Following Lemma~\ref{l:sufficient_conditions_new}, there exist terms $t_i(X)\in\T_{P_i}(S^n,\Gamma)$ such that the solution set of the equation $t_i(X)=(1,1,1)$ equals $S^n\setminus\{P_i\}$. Thus, the solution set of the system $\Ss=\{t_i(X)=(1,1,1)|i\in \mathcal{I}\}$ coincides with $M$.
\end{proof}

The information of the author:

Artem N. Shevlyakov

Omsk Branch of Institute of Mathematics, Siberian Branch of the Russian Academy of Sciences

644099 Russia, Omsk, Pevtsova st. 13

Phone: +7-3812-23-25-51.

e-mail: \texttt{a\_shevl@mail.ru}
\end{document}